\documentclass[10pt,reqno]{amsart}
\usepackage{amssymb,latexsym}
\usepackage{graphicx}
\usepackage{fancyhdr}
\usepackage{enumerate}
\usepackage{mathtools}

\newcommand{\R}{{\mathbb R}}

\newcommand{\C}{{\mathbb C}}

\newcommand{\B}{{\textrm{BC}}}
\newcommand{\p}{{\textrm{P}}}

\DeclarePairedDelimiter\abs{\lvert}{\rvert}
\newcommand{\mymod}[3]{#1 \equiv #2 \kern -0.5em \pmod{#3}}
\newcommand{\mynotmod}[3]{#1 \not \equiv #2 \kern -0.6em \pmod{#3}}

\theoremstyle{plain}
\numberwithin{equation}{section}
\newtheorem{thm}{Theorem}[section]
\newtheorem{theorem}[thm]{Theorem}

\newtheorem{lemma}[thm]{Lemma}
\newtheorem{definition}[thm]{Definition}
\newtheorem{proposition}[thm]{Proposition}

\begin{document}

\setcounter{page}{1}

\title[Bicomplex Third-order Jacobsthal Quaternions]{Bicomplex Third-order Jacobsthal Quaternions}
\author{Gamaliel Cerda}
\address{Institute of Mathematics\\
                Pontificia Universidad Cat\'olica de Valpara\'iso\\
                Blanco Viel 596, Cerro Bar\'on\\
                Valpara\'iso, Chile}
\email{gamaliel.cerda.m@mail.pucv.cl}


\begin{abstract}
The aim of this work is to consider the bicomplex third-order Jacobsthal quaternions and to present some properties involving this sequence, including the Binet-style formulae and the generating functions. Furthermore, Cassini's identity and d'Ocagne's identity for this type of bicomplex quaternions are given, and a different way to find the $n$-th term of this sequence is stated using the determinant of a four-diagonal matrix whose entries are bicomplex third-order quaternions.
\end{abstract}

\maketitle

\section{Introduction}
The set of bicomplex numbers, denoted by $\mathbb{BC}$, forms a two-dimensional algebra over $\C$, and since $\C$ is of dimension two over $\R$, the bicomplex numbers are an algebra over $\R$ of dimension four. The bicomplex numbers are defined by the basis $1$, $i$, $j$, $ij$, where $i$, $j$, and $ij$ satisfy the properties:
\begin{equation}\label{e0}
i^{2}=-1,\ j^{2}=-1,\ ij=ji.
\end{equation}
A bicomplex number $w$ can be expressed as follows:
\begin{equation}\label{e1}
\begin{aligned}
w&=z_{1}+jz_{2}\ \ \ \ (z_{1},z_{2}\in \C) \\
&=(x_{1}+iy_{1})+j(x_{2}+iy_{2}),
\end{aligned}
\end{equation}
where $x_{1}, x_{2}, y_{1}, y_{2}\in \R$.

The bicomplex numbers share some structures and properties of the complex numbers $\C$, but there are differences between them (see, \cite{Lu1,Lu2}). Bicomplex numbers forms a commutative ring with unity which contain the complex numbers. Furthermore, for any $w=x_{1}+iy_{1}+jx_{2}+ijy_{2}$ and $w'=x_{1}'+iy_{1}'+jx_{2}'+ijy_{2}'$, bicomplex addition and product are defined by the following:
\begin{equation}\label{e2}
w+w'=(x_{1}+x_{1}')+i(y_{1}+y_{1}')+j(x_{2}+x_{2}')+ij(y_{2}+y_{2}')
\end{equation}
and
\begin{equation}\label{e3}
\begin{aligned}
w\times w'&=(x_{1}+iy_{1}+jx_{2}+ijy_{2})\times (x_{1}'+iy_{1}'+jx_{2}'+ijy_{2}')\\
&=(x_{1}x_{1}'-y_{1}y_{1}'-x_{2}x_{2}'+y_{2}y_{2}')\\
&\ \ +i(x_{1}y_{1}'+y_{1}x_{1}'-x_{2}y_{2}'-y_{2}x_{2}')\\
&\ \ +j(x_{1}x_{2}'-y_{1}y_{2}'+x_{2}x_{1}'-y_{2}y_{1}')\\
&\ \ +ij(x_{1}y_{2}'+y_{1}x_{2}'+x_{2}y_{1}'+y_{2}x_{1}'),
\end{aligned}
\end{equation}
respectively.

The set of bicomplex numbers $\mathbb{BC}$ is a real vector space with this addition and the multiplication of a bicomplex number by a real scalar, and with the bicomplex number product, $\times$ is a real associative algebra. In addition, the vector space with the properties of scalar multiplication and the product of the bicomplex numbers is a commutative algebra. For more details about these type of numbers, see, for example, \cite{Ro-Sha}.

There are three different conjugations for bicomplex numbers as follows:
$$w_{i}^*=\overline{z_{1}}+j\overline{z_{2}}=x_{1}-iy_{1}+jx_{2}-ijy_{2},$$
$$w_{j}^*=z_{1}-jz_{2}=x_{1}+iy_{1}-jx_{2}-ijy_{2},$$
$$w_{ij}^*=\overline{z_{1}}-j\overline{z_{2}}=x_{1}-iy_{1}-jx_{2}+ijy_{2},$$
and the squares of norms of the bicomplex numbers are given by
$$Nr^{2}(w)_{i}=\mid w\times w_{i}^* \mid =\mid x_{1}^{2}+y_{1}^{2}-x_{2}^{2}-y_{2}^{2}+2j(x_{1}x_{2}+y_{1}y_{2}) \mid,$$
$$Nr^{2}(w)_{j}=\mid w\times w_{j}^* \mid =\mid x_{1}^{2}+x_{2}^{2}-y_{1}^{2}-y_{2}^{2}+2i(x_{1}y_{1}+x_{2}y_{2}) \mid,$$
$$Nr^{2}(w)_{ij}=\mid w\times w_{ij}^* \mid =\mid x_{1}^{2}+y_{1}^{2}+x_{2}^{2}+y_{2}^{2}+2ij(x_{1}y_{2}-x_{2}y_{1}) \mid.$$

In general, quaternions were formally introduced by W. R. Hamilton in 1843 and some background about this type of hypercomplex numbers can be found, for example, in \cite{Co-Sm,Wa}. The field $\mathbb{H}$ of quaternions is a four-dimensional non-commutative $\R$--field generated by four base elements $1$, $i$, $j$ and $k$ that satisfy the following rules:
\begin{equation}\label{e4}
i^{2}=j^{2}=k^{2}=-1,\ ij=-ji=k,\ jk=-kj=i, \ ki=-ik=j.
\end{equation}
Quaternions and bicomplex numbers are generalizations of complex numbers, but one difference between them is that quaternions are non-commutative, whereas bicomplex numbers are commutative. Similarly, considering bicomplex quaternions, they can also be defined by four base elements $1$, $i$, $j$, and $ij$ that satisfy the rules (\ref{e0}).

In \cite{Ce1}, Cerda-Morales introduced the third-order Jacobsthal quaternion sequence and such sequences have been studied in several papers (see, for example, \cite{Co-Ba}). Also generalizations of the third-order Jacosbtha quaternions have been presented in the literature (see, for example, \cite{Ce2}). Recall that, the sequence of third-order Jacobsthal numbers $\{J_{n}^{(3)}\}_{n=0}^{\infty}$ is defined by the following:
\begin{equation}\label{e5}
J_{n}^{(3)}=J_{n-1}^{(3)}+J_{n-2}^{(3)}+2J_{n-3}^{(3)},\ n\geq 3
\end{equation}
with the initial terms $J_{0}^{(3)}=0$ and $J_{1}^{(3)}=J_{2}^{(3)}=1$.

The Binet-style formula for this sequences is given by the following:
\begin{equation}\label{e6}
J_{n}^{(3)}=\frac{1}{7}\left(2^{n+1}-\left(\frac{A\omega_{1}^{n}-B\omega_{2}^{n}}{\omega_{1}-\omega_{2}}\right)\right)=\frac{1}{7}\left(2^{n+1}-V_{n}^{(3)}\right),
\end{equation}
where $A=-3-2\omega_{2}$, $B=-3-2\omega_{1}$ and $2$, $\omega_{1}$ and $\omega_{2}$ are the roots of the characteristic equation $x^{3}-x^{2}-x-2=0$ associated with the above recurrence relation (\ref{e5}). Here, the sequence $\{V_{n}^{(3)}\}_{n\geq 0}$ is defined by
\begin{equation}\label{ev}
V_{n}^{(3)}=\frac{A\omega_{1}^{n}-B\omega_{2}^{n}}{\omega_{1}-\omega_{2}}=\left\{ 
\begin{array}{ccc}
2 & \textrm{if} & \mymod{n}{0}{3} \\ 
-3 & \textrm{if} & \mymod{n}{1}{3} \\ 
1 & \textrm{if} & \mymod{n}{2}{3}%
\end{array}%
\right. .
\end{equation}
Note that $\omega_{1}+\omega_{2}=-1$ and $\omega_{1}\omega_{2}=1$. Furthermore, $V_{n+2}^{(3)}+V_{n+1}^{(3)}+V_{n}^{(3)}=0$ and $V_{n}^{(3)}=V_{n+3}^{(3)}$ for all $n\geq 0$.

The more recent research in the topic of sequences of bicomplex quaternions is the work of Ayd\i n in \cite{Ay} about the bicomplex Fibonacci quaternions and Catarino in \cite{Ca} about bicomplex $k$-Pell quaternions. Motivated essentially by this works, in this paper, we introduce the bicomplex third-order Jacobsthal quaternions and we obtain some properties, including the respective Binet-style formulae, generating functions and some other identities.

\section{The Bicomplex Third-order Jacobsthal Quaternions and Some Basic Properties}
Let $1$, $i$, $j$ and $ij$ satisfy the rules (\ref{e0}).
\begin{definition}\label{d1}
The bicomplex third-order Jacobsthal quaternions $\{\emph{\B}_{J,n}^{(3)}\}_{n=0}^{\infty}$ are defined by the following:
\begin{equation}\label{e7}
\emph{\B}_{J,n}^{(3)}=J_{n}^{(3)}+iJ_{n+1}^{(3)}+jJ_{n+2}^{(3)}+ijJ_{n+3}^{(3)},
\end{equation}
where $J_{n}^{(3)}$ is the $n$-th third-order Jacobsthal number.
\end{definition} 

From Definition \ref{d1} and the use of (\ref{e5}), we easily show that $\{\B_{J,n}^{(3)}\}_{n=0}^{\infty}$ can also be defined by the recurrence relation:
\begin{equation}\label{e8}
\B_{J,n}^{(3)}=\B_{J,n-1}^{(3)}+\B_{J,n-2}^{(3)}+2\B_{J,n-3}^{(3)},\ n\geq3
\end{equation}
with the initial conditions $\B_{J,0}^{(3)}=i+j+2ij$, $\B_{J,1}^{(3)}=1+i+2j+5ij$ and $\B_{J,2}^{(3)}=1+2i+5j+9ij$.

Let $n, m\geq 0$. For two bicomplex third-order Jacobsthal quaternions $\B_{J,n}^{(3)}$ and $\B_{J,m}^{(3)}$, addition and subtraction are obviously defined by the following:
\begin{equation}\label{e9}
\begin{aligned}
\B_{J,n}^{(3)}\pm \B_{J,m}^{(3)}&=\left(J_{n}^{(3)}\pm J_{m}^{(3)}\right)+i\left(J_{n+1}^{(3)}\pm J_{m+1}^{(3)}\right)\\
&\ \ \ + j\left(J_{n+2}^{(3)}\pm J_{m+2}^{(3)}\right)+ij\left(J_{n+3}^{(3)}\pm J_{m+3}^{(3)}\right),
\end{aligned}
\end{equation}
and multiplication by
\begin{equation}\label{e10}
\begin{aligned}
\B_{J,n}^{(3)}\times \B_{J,m}^{(3)}&=\left(J_{n}^{(3)}J_{m}^{(3)}-J_{n+1}^{(3)}J_{m+1}^{(3)}-J_{n+2}^{(3)}J_{m+2}^{(3)}+J_{n+3}^{(3)}J_{m+3}^{(3)}\right)\\
&\ \ +i\left(J_{n}^{(3)}J_{m+1}^{(3)}+J_{n+1}^{(3)}J_{m}^{(3)}-J_{n+2}^{(3)}J_{m+3}^{(3)}-J_{n+3}^{(3)}J_{m+2}^{(3)}\right)\\
&\ \ +j\left(J_{n}^{(3)}J_{m+2}^{(3)}-J_{n+1}^{(3)}J_{m+3}^{(3)}+J_{n+2}^{(3)}J_{m}^{(3)}-J_{n+3}^{(3)}J_{m+1}^{(3)}\right)\\
&\ \ +ij\left(J_{n}^{(3)}J_{m+3}^{(3)}+J_{n+1}^{(3)}J_{m+2}^{(3)}+J_{n+2}^{(3)}J_{m+1}^{(3)}+J_{n+3}^{(3)}J_{m}^{(3)}\right)\\
&=\B_{J,m}^{(3)}\times \B_{J,n}^{(3)}.
\end{aligned}
\end{equation}

The multiplication of a bicomplex third-order Jacobsthal quaternion by the real scalar $\lambda$ is defined by the following:
$$\lambda \B_{J,n}^{(3)}=\lambda J_{n}^{(3)}+i\lambda J_{n+1}^{(3)}+j\lambda J_{n+2}^{(3)}+ij\lambda J_{n+3}^{(3)}.$$

The different conjugations for bicomplex k-Pell quaternions are presented as follows:
\begin{equation}\label{n1}
\left(\B_{J,n}^{(3)}\right)_{i}^*=J_{n}^{(3)}-iJ_{n+1}^{(3)}+jJ_{n+2}^{(3)}-ijJ_{n+3}^{(3)},
\end{equation}
\begin{equation}\label{n2}
\left(\B_{J,n}^{(3)}\right)_{j}^*=J_{n}^{(3)}+iJ_{n+1}^{(3)}-jJ_{n+2}^{(3)}-ijJ_{n+3}^{(3)},
\end{equation}
\begin{equation}\label{n3}
\left(\B_{J,n}^{(3)}\right)_{ij}^*=J_{n}^{(3)}-iJ_{n+1}^{(3)}-jJ_{n+2}^{(3)}+ijJ_{n+3}^{(3)}.
\end{equation}

Using these conjugations, we have two basic properties of the bicomplex third-order Jacobsthal quaternions.
\begin{lemma}
For any integer numbers $n, m\geq 0$, the following relations between the conjugate of these bicomplex third-order Jacobsthal quaternions are true:
\begin{equation}\label{l1}
\left(\emph{\p}_{J,n,m}^{(3)}\right)_{i}^*=\left(\emph{\B}_{J,n}^{(3)}\right)_{i}^*\times \left(\emph{\B}_{J,m}^{(3)}\right)_{i}^*=\left(\emph{\B}_{J,m}^{(3)}\right)_{i}^*\times \left(\emph{\B}_{J,n}^{(3)}\right)_{i}^*,
\end{equation}
\begin{equation}\label{l2}
\left(\emph{\p}_{J,n,m}^{(3)}\right)_{j}^*=\left(\emph{\B}_{J,n}^{(3)}\right)_{j}^*\times \left(\emph{\B}_{J,m}^{(3)}\right)_{j}^*=\left(\emph{\B}_{J,m}^{(3)}\right)_{j}^*\times \left(\emph{\B}_{J,n}^{(3)}\right)_{j}^*,
\end{equation}
\begin{equation}\label{l3}
\left(\emph{\p}_{J,n,m}^{(3)}\right)_{ij}^*=\left(\emph{\B}_{J,n}^{(3)}\right)_{ij}^*\times \left(\emph{\B}_{J,m}^{(3)}\right)_{ij}^*=\left(\emph{\B}_{J,m}^{(3)}\right)_{ij}^*\times \left(\emph{\B}_{J,n}^{(3)}\right)_{ij}^*,
\end{equation}
where $\emph{\p}_{J,n,m}^{(3)}=\emph{\B}_{J,n}^{(3)}\times \emph{\B}_{J,m}^{(3)}$.
\end{lemma}
\begin{proof}
We can prove these equalities using (\ref{n1}), (\ref{n2}), and (\ref{n3}), and the multiplication of two bicomplex third-orderJacobsthal quaternions in (\ref{e10}).
\end{proof}

\begin{lemma}
For any positive integer number $n$, the squares of norms in different ways of the bicomplex third-order Jacobsthal quaternions are given by the following:
\begin{equation}\label{l4}
Nr^{2}\left(\emph{\B}_{J,n}^{(3)}\right)_{i}= \abs[\Big]{\emph{\B}_{J,n}^{(3)} \times \left(\emph{\B}_{J,n}^{(3)}\right)_{i}^*}=\abs[\Big]{X_{1}+2jY_{1}},
\end{equation}
\begin{equation}\label{l5}
Nr^{2}\left(\emph{\B}_{J,n}^{(3)}\right)_{j}= \abs[\Big]{\emph{\B}_{J,n}^{(3)} \times \left(\emph{\B}_{J,n}^{(3)}\right)_{j}^*}=\abs[\Big]{X_{2}+2iY_{2}},
\end{equation}
\begin{equation}\label{l6}
Nr^{2}\left(\emph{\B}_{J,n}^{(3)}\right)_{ij}= \abs[\Big]{\emph{\B}_{J,n}^{(3)} \times \left(\emph{\B}_{J,n}^{(3)}\right)_{ij}^*}=\abs[\Big]{X_{3}+2ijY_{3}},
\end{equation}
where 
\begin{align*}
X_{1}&=\frac{1}{49}\left\lbrace
\begin{array}{c}
-75\cdot 2^{2n+2} -3\cdot 2^{n+2}\left(2V_{n+1}^{(3)}-3V_{n}^{(3)}\right)-2V_{n}^{(3)}V_{n+1}^{(3)}-\left(V_{n}^{(3)}\right)^{2}
\end{array}
\right\rbrace,\\
Y_{1}&=\frac{1}{49}\left\lbrace
\begin{array}{c}
5\cdot 2^{2n+4} -2^{n+1}\left(2V_{n+1}^{(3)}-5V_{n+2}^{(3)}\right)-\left(V_{n}^{(3)}\right)^{2}
\end{array}
\right\rbrace,\\
X_{2}&=\frac{1}{49}\left\lbrace
\begin{array}{c}
-51\cdot 2^{2n+2} -2^{n+2}\left(6V_{n+2}^{(3)}-5V_{n}^{(3)}\right)+2V_{n}^{(3)}V_{n+1}^{(3)}+\left(V_{n}^{(3)}\right)^{2}
\end{array}
\right\rbrace,\\
Y_{2}&=\frac{1}{49}\left\lbrace
\begin{array}{c}
17\cdot 2^{2n+3} -2^{n+1}\left(5V_{n}^{(3)}+7V_{n+2}^{(3)}\right)-\left(V_{n}^{(3)}\right)^{2}
\end{array}
\right\rbrace,\\
X_{3}&=\frac{1}{49}\left\lbrace
\begin{array}{c}
85\cdot 2^{2n+2} -2^{n+2}\left(2V_{n+2}^{(3)}+7V_{n}^{(3)}\right)-2V_{n+1}^{(3)}V_{n+2}^{(3)}+3\left(V_{n}^{(3)}\right)^{2}
\end{array}
\right\rbrace,\\
Y_{3}&=\frac{1}{49}\left\lbrace
\begin{array}{c}
 -2^{n+1}\left(11V_{n}^{(3)}-2V_{n+1}^{(3)}\right)+\left(V_{n}^{(3)}\right)^{2}-V_{n+1}^{(3)}V_{n+2}^{(3)}
\end{array}
\right\rbrace.
\end{align*}
\end{lemma}
\begin{proof}
These equalities can easily be proved using Definition \ref{d1} and of norm of bicomplex third-order Jacobsthal quaternions taking into account the three different conjugations given in (\ref{n1}), (\ref{n2}) and (\ref{n3}), considering $z_{1}=J_{n}^{(3)}+iJ_{n+1}^{(3)}$ and $z_{2}=J_{n+2}^{(3)}+iJ_{n+3}^{(3)}$. We present the proof of the first equality. For this, by taking into account that $J_{n}^{(3)}=\frac{1}{7}\left(2^{n+1}-V_{n}^{(3)}\right)$ and sequence $V_{n}^{(3)}$ as in (\ref{ev}), we obtain $\B_{J,n}^{(3)} \times \left(\B_{J,n}^{(3)}\right)_{i}^*=X_{1}+2jY_{1}$, where
\begin{align*}
X_{1}&=\left\lbrace
\begin{array}{c}
\left(J_{n}^{(3)}\right)^{2}+\left(J_{n+1}^{(3)}\right)^{2}-\left(J_{n+2}^{(3)}\right)^{2}-\left(J_{n+3}^{(3)}\right)^{2}
\end{array}
\right\rbrace\\
&=\frac{1}{49}\left\lbrace
\begin{array}{c}
\left(2^{n+1}-V_{n}^{(3)}\right)^{2}+\left(2^{n+2}-V_{n+1}^{(3)}\right)^{2}\\
-\left(2^{n+3}-V_{n+2}^{(3)}\right)^{2}-\left(2^{n+4}-V_{n+3}^{(3)}\right)^{2}
\end{array}
\right\rbrace\\
&=\frac{1}{49}\left\lbrace
\begin{array}{c}
-75\cdot 2^{2n+2} -2^{n+2}\left(V_{n}^{(3)}+2V_{n+1}^{(3)}-4V_{n+2}^{(3)}-8V_{n+3}^{(3)}\right)\\
+\left(V_{n}^{(3)}\right)^{2}+\left(V_{n+1}^{(3)}\right)^{2}-\left(V_{n+2}^{(3)}\right)^{2}-\left(V_{n+3}^{(3)}\right)^{2}
\end{array}
\right\rbrace\\
&=\frac{1}{49}\left\lbrace
\begin{array}{c}
-75\cdot 2^{2n+2} -3\cdot 2^{n+2}\left(2V_{n+1}^{(3)}-3V_{n}^{(3)}\right)-2V_{n}^{(3)}V_{n+1}^{(3)}-\left(V_{n}^{(3)}\right)^{2}
\end{array}
\right\rbrace\\
\end{align*}
and
\begin{align*}
Y_{1}&=J_{n}^{(3)}J_{n+2}^{(3)}+J_{n+1}^{(3)}J_{n+3}^{(3)}\\
&=\frac{1}{49}\left\lbrace
\begin{array}{c}
\left(2^{n+1}-V_{n}^{(3)}\right)\left(2^{n+3}-V_{n+2}^{(3)}\right)+\left(2^{n+2}-V_{n+1}^{(3)}\right)\left(2^{n+4}-V_{n+3}^{(3)}\right)
\end{array}
\right\rbrace\\
&=\frac{1}{49}\left\lbrace
\begin{array}{c}
5\cdot 2^{2n+4} -2^{n+1}\left(V_{n+2}^{(3)}+2V_{n+3}^{(3)}+4V_{n}^{(3)}+8V_{n+1}^{(3)}\right)\\
+V_{n}^{(3)}V_{n+2}^{(3)}+V_{n+1}^{(3)}V_{n+3}^{(3)}
\end{array}
\right\rbrace\\
&=\frac{1}{49}\left\lbrace
\begin{array}{c}
5\cdot 2^{2n+4} -2^{n+1}\left(2V_{n+1}^{(3)}-5V_{n+2}^{(3)}\right)-\left(V_{n}^{(3)}\right)^{2}
\end{array}
\right\rbrace,
\end{align*}
and the result follows for $Nr^{2}\left(\B_{J,n}^{(3)}\right)_{i}$. In a similar way, we can find the expressions of $Nr^{2}\left(\B_{J,n}^{(3)}\right)_{j}$ and $Nr^{2}\left(\B_{J,n}^{(3)}\right)_{ij}$.
\end{proof}

The sum of the first $n$ terms of the bicomplex third-order Jacobsthal quaternions sequence is stated in the next result.
\begin{theorem}\label{t1}
The sum of the first $n$ terms of the bicomplex third-order Jacobsthal quaternions sequence is given by the following:
\begin{equation}\label{ee0}
\begin{aligned}
&\sum_{s=0}^{n}\emph{\B}_{J,s}^{(3)}=S_{J}(n)+iS_{J}(n+1)+j\left(S_{J}(n+2)-1\right)+ij\left(S_{J}(n+3)-2\right)\\
&=\left\{ 
\begin{array}{ccc}
J_{n+1}^{(3)}-1 +i\left(J_{n+2}^{(3)}\right)+j\left(J_{n+3}^{(3)}-1\right)+ij\left(J_{n+4}^{(3)}-3\right)& \textrm{if} &\mymod{n}{0}{3} \\ 
J_{n+1}^{(3)} +i\left(J_{n+2}^{(3)}\right)+j\left(J_{n+3}^{(3)}-2\right)+ij\left(J_{n+4}^{(3)}-2\right)& \textrm{if} & \mymod{n}{1}{3} \\ 
J_{n+1}^{(3)} +i\left(J_{n+2}^{(3)}-1\right)+j\left(J_{n+3}^{(3)}-1\right)+ij\left(J_{n+4}^{(3)}-2\right)&  \textrm{if} & \mymod{n}{2}{3} 
\end{array}
\right. ,
\end{aligned}
\end{equation}
where $S_{J}(n)=\frac{1}{3}\left(J_{n+2}^{(3)}+2J_{n}^{(3)}-1\right)$ denotes the sum of the first $n$ terms of third-order Jacobsthal sequence, which is given by \cite[Page 32]{Co-Ba}.
\end{theorem}
\begin{proof}
By the use of Definition \ref{d1}, we have: 
\begin{align*}
\sum_{s=0}^{n}\B_{J,s}^{(3)}&=\sum_{s=0}^{n}J_{s}^{(3)}+i\sum_{s=0}^{n}J_{s+1}^{(3)}+j\sum_{s=0}^{n}J_{s+2}^{(3)}+ij\sum_{s=0}^{n}J_{s+3}^{(3)}\\
&=\sum_{s=0}^{n}J_{s}^{(3)}+i\sum_{s=1}^{n+1}J_{s}^{(3)}+j\sum_{s=2}^{n+2}J_{s}^{(3)}+ij\sum_{s=3}^{n+3}J_{s}^{(3)}\\
&=\sum_{s=0}^{n}J_{s}^{(3)}+i\left(\sum_{s=0}^{n+1}J_{s}^{(3)}\right)+j\left(\sum_{s=0}^{n+2}J_{s}^{(3)}-1\right)+ij\left(\sum_{s=0}^{n+3}J_{s}^{(3)}-2\right)\\
\end{align*}
Now, taking into account the result stated in \cite[Page 32]{Co-Ba}, $$\sum_{0}^{n}J_{s}^{(3)}=\frac{1}{3}\left(J_{n+2}^{(3)}+2J_{n}^{(3)}-1\right)=\left\{ 
\begin{array}{ccc}
J_{n+1}^{(3)}-1 & \textrm{if} & \mymod{n}{0}{3} \\ 
J_{n+1}^{(3)} & \textrm{if} & \mynotmod{n}{0}{3} 
\end{array}
\right. ,$$ and the initial condition $J_{0}^{(3)}=0$, the result easily follows.
\end{proof}

\section{Generating functions and Binet's formula}
Next, we shall give the generating functions for the bicomplex third-order Jacobsthal quaternions sequence. We shall write such sequence as a power series where each term of the sequence corresponds to coefficients of the series. Consider the bicomplex third-order Jacobsthal quaternions sequence $\{\B_{J,n}^{(3)}\}_{n=0}^{\infty}$. By definition of generating function of a sequence, considering this sequence, the associated generating function $g_{\B_{J,n}^{(3)}}(t)$ is defined, respectively, by the following:
\begin{equation}\label{g1}
g_{\B_{J,n}^{(3)}}(t)=\sum_{n=0}^{\infty}\B_{J,n}^{(3)}t^{n}
\end{equation}

Therefore, using (\ref{g1}), we obtain the following result:
\begin{theorem}\label{t2}
The generating function for the bicomplex third-order quaternions sequence is given by the following:
\begin{align*}
g_{\emph{\B}_{J,n}^{(3)}}(t)&=\frac{\emph{\B}_{J,0}^{(3)}+\left(\emph{\B}_{J,1}^{(3)}-\emph{\B}_{J,0}^{(3)}\right)t+\left(\emph{\B}_{J,2}^{(3)}-\emph{\B}_{J,1}^{(3)}-\emph{\B}_{J,0}^{(3)}\right)t^{2}}{1-t-t^{2}-2t^{3}}\\
\end{align*}
\end{theorem}
\begin{proof}
Using (\ref{g1}), we have
\begin{equation}\label{g2}
g_{\B_{J,n}^{(3)}}(t)=\B_{J,0}^{(3)}+\B_{J,1}^{(3)}t+\B_{J,2}^{(3)}t^{2}+\cdots +\B_{J,n}^{(3)}t^{n}+\cdots .
\end{equation}
Multiplying both sides of (\ref{g2}) by $-t$, we obtain
\begin{equation}\label{g3}
-tg_{\B_{J,n}^{(3)}}(t)=-\B_{J,0}^{(3)}t-\B_{J,1}^{(3)}t^{2}-\B_{J,2}^{(3)}t^{3}-\cdots -\B_{J,n}^{(3)}t^{n+1}+\cdots .
\end{equation}
Now, multiplying both sides of (\ref{g2}) by $-t^{2}$ and $-2t^{3}$, we get
\begin{equation}\label{g4}
-t^{2}g_{\B_{J,n}^{(3)}}(t)=-\B_{J,0}^{(3)}t^{2}-\B_{J,1}^{(3)}t^{3}-\B_{J,2}^{(3)}t^{4}-\cdots -\B_{J,n}^{(3)}t^{n+2}+\cdots 
\end{equation}
and
\begin{equation}\label{g5}
-2t^{3}g_{\B_{J,n}^{(3)}}(t)=-2\B_{J,0}^{(3)}t^{3}-2\B_{J,1}^{(3)}t^{4}-2\B_{J,2}^{(3)}t^{5}-\cdots -2\B_{J,n}^{(3)}t^{n+3}+\cdots ,
\end{equation}
respectively. Adding (\ref{g2})-(\ref{g5}), and using (\ref{e8}), we have
$$(1-t-t^{2}-2t^{3})g_{\B_{J,n}^{(3)}}(t)=\B_{J,0}^{(3)}+\left(\B_{J,1}^{(3)}-\B_{J,0}^{(3)}\right)t+\left(\B_{J,2}^{(3)}-\B_{J,1}^{(3)}-\B_{J,0}^{(3)}\right)t^{2},$$
and the result follows.
\end{proof}

The following result, with easy proof, using Binet's formula of $J_{n}^{(3)}$ given by (\ref{e6}) and it will be useful in the statement of the Binet formula of $\B_{J,n}^{(3)}$.
\begin{lemma}\label{lem3}
Let $\{J_{n}^{(3)}\}_{n\geq0}$, $2$, $\omega_{1}$ and $\omega_{2}$ be as above. Then, we have for all integer $n\geq 0$
\begin{equation}\label{ap1}
\textrm{Quadratic app. of $\{J_{n}^{(3)}\}$}: \left\{
\begin{array}{c }
P2^{n+2}=2^{2} J_{n+2}^{(3)}+2(J_{n+1}^{(3)}+2J_{n}^{(3)})+ 2J_{n+1}^{(3)},\\
Q\omega_{1}^{n+2}=\omega_{1}^{2} J_{n+2}^{(3)}+\omega_{1}(J_{n+1}^{(3)}+2J_{n}^{(3)})+ 2J_{n+1}^{(3)},\\
R\omega_{2}^{n+2}=\omega_{2}^{2} J_{n+2}^{(3)}+\omega_{2}(J_{n+1}^{(3)}+2J_{n}^{(3)})+ 2J_{n+1}^{(3)},
\end{array}
\right.
\end{equation}
where $P=1-(\omega_{1}+\omega_{2})$, $Q=1-(\alpha+\omega_{2})$ and $R=1-(\alpha+\omega_{1})$. 
\end{lemma}

The next result will be also used in the statement of the Binet formula for the bicomplex third-order Jacobsthal quaternion sequence.
\begin{theorem}\label{t3}
For the generating function given in Theorem \ref{t2}, we have
$$
g_{\emph{\B}_{J,n}^{(3)}}(t)=\frac{1}{\phi(J)}\left\lbrace
\begin{array}{c}
\left(\omega_{1}-\omega_{2}\right) \left(\frac{\B_{J,2}^{(3)}+\B_{J,1}^{(3)}+\B_{J,0}^{(3)}}{1-2t}\right)\\
-\left(2-\omega_{2}\right) \left(\frac{\B_{J,2}^{(3)}+(\omega_{1}-1)\B_{J,1}^{(3)}+(\omega_{1}^{2}-\omega_{1}-1)\B_{J,0}^{(3)}}{1-\omega_{1}t}\right)\\
+\left(2-\omega_{1}\right) \left(\frac{\B_{J,2}^{(3)}+(\omega_{2}-1)\B_{J,1}^{(3)}+(\omega_{2}^{2}-\omega_{2}-1)\B_{J,0}^{(3)}}{1-\omega_{2}t}\right)
\end{array}
\right\rbrace,
$$
where $\phi(J)=(2-\omega_{1})(2-\omega_{2})(\omega_{1}-\omega_{2})$.
\end{theorem}
\begin{proof}
From the expression $g_{\B_{J,n}^{(3)}}(t)$ given by Theorem \ref{t2} and by the use of the fact that $\omega_{1}+\omega_{2}=-1$ and $\omega_{1}\omega_{2}=1$, we have
\begin{align*}
g_{\B_{J,n}^{(3)}}(t)&=\frac{\B_{J,0}^{(3)}+\left(\B_{J,1}^{(3)}-\B_{J,0}^{(3)}\right)t+\left(\B_{J,2}^{(3)}-\B_{J,1}^{(3)}-\B_{J,0}^{(3)}\right)t^{2}}{1-t-t^{2}-2t^{3}}\\
&=\frac{\B_{J,0}^{(3)}+\left(\B_{J,1}^{(3)}-\B_{J,0}^{(3)}\right)t+\left(\B_{J,2}^{(3)}-\B_{J,1}^{(3)}-\B_{J,0}^{(3)}\right)t^{2}}{1-(2+\omega_{1}+\omega_{2})t+(2\omega_{1}+2\omega_{2}+\omega_{1}\omega_{2})t^{2}-2\omega_{1}\omega_{2}t^{3}}\\
&=\frac{\B_{J,0}^{(3)}+\left(\B_{J,1}^{(3)}-\B_{J,0}^{(3)}\right)t+\left(\B_{J,2}^{(3)}-\B_{J,1}^{(3)}-\B_{J,0}^{(3)}\right)t^{2}}{(1-2t)(1-\omega_{1}t)(1-\omega_{2}t)}.
\end{align*}
Now, multiplying and dividing the right side of this expression by $(2-\omega_{1})(2-\omega_{2})$, $(2-\omega_{1})(\omega_{1}-\omega_{2})$ and $(2-\omega_{2})(\omega_{1}-\omega_{2})$, respectively, the result follows.
\end{proof}

The next result gives the Binet formula for the bicomplex third-order Jacobsthal quaternion sequence.
\begin{theorem}\label{t4}
For $n\geq 0$, we have
$$\emph{\B}_{J,n}^{(3)}=\frac{1}{\phi(J)}\left(\left(\omega_{1}-\omega_{2}\right) \widehat{2}2^{n+1}-\left(2-\omega_{2}\right) \widehat{\omega_{1}}\omega_{1}^{n+1}+\left(2-\omega_{1}\right) \widehat{\omega_{2}}\omega_{2}^{n+1}\right),$$
where $\widehat{2}$, $\widehat{\omega_{1}}$ and $\widehat{\omega_{2}}$ are bicomplex quaternions defined by $\widehat{2}=1+2i+4j+8ij$, $\widehat{\omega_{1}}=1+i\omega_{1}+j\omega_{1}^{2}+ij\omega_{1}^{3}$ and $\widehat{\omega_{2}}=1+i\omega_{2}+j\omega_{2}^{2}+ij\omega_{2}^{3}$, respectively.
\end{theorem}
\begin{proof}
Using (\ref{e5}), the Binet formula for the third-order Jacobsthal numbers considered in (\ref{e6}) and the Definition \ref{d1}, we have
\begin{align*}
&g_{\B_{J,n}^{(3)}}(t)\\
&=\frac{1}{\phi(J)}\left\lbrace
\begin{array}{c}
\left(\omega_{1}-\omega_{2}\right) \left(\B_{J,2}^{(3)}+\B_{J,1}^{(3)}+\B_{J,0}^{(3)}\right)\sum_{n=0}^{\infty}2^{n}t^{n}\\
-\left(2-\omega_{2}\right) \left(\B_{J,2}^{(3)}+(\omega_{1}-1)\B_{J,1}^{(3)}+(\omega_{1}^{2}-\omega_{1}-1)\B_{J,0}^{(3)}\right)\sum_{n=0}^{\infty}\omega_{1}^{n}t^{n}\\
+\left(2-\omega_{1}\right) \left(\B_{J,2}^{(3)}+(\omega_{2}-1)\B_{J,1}^{(3)}+(\omega_{2}^{2}-\omega_{2}-1)\B_{J,0}^{(3)}\right)\sum_{n=0}^{\infty}\omega_{2}^{n}t^{n}
\end{array}
\right\rbrace\\
&=\frac{1}{\phi(J)}\left\lbrace
\begin{array}{c}
\left(\omega_{1}-\omega_{2}\right) \widehat{2}\sum_{n=0}^{\infty}2^{n+1}t^{n}-\left(2-\omega_{2}\right) \widehat{\omega_{1}}\sum_{n=0}^{\infty}\omega_{1}^{n+1}t^{n}\\
+\left(2-\omega_{1}\right) \widehat{\omega_{2}}\sum_{n=0}^{\infty}\omega_{2}^{n+1}t^{n}
\end{array}
\right\rbrace\\
&=\frac{1}{\phi(J)}\sum_{n=0}^{\infty}\left\lbrace
\begin{array}{c}
\left(\omega_{1}-\omega_{2}\right) \widehat{2}2^{n+1}-\left(2-\omega_{2}\right) \widehat{\omega_{1}}\omega_{1}^{n+1}+\left(2-\omega_{1}\right) \widehat{\omega_{2}}\omega_{2}^{n+1}
\end{array}
\right\rbrace t^{n},
\end{align*}
and by (\ref{g1}), the result easily follows.
\end{proof}

\section{Some Identities Involving These Sequences} 
As a consequence of the Binet formulae of Theorem \ref{t4}, we get, for the sequence of bicomplex third-order Jacobsthal quaternions, the following interesting identities.
\begin{proposition}[d'Ocagne's identity]
Suppose that $n$ is a non-negative integer number and $m$ any natural number. If $m > n$ and $\emph{\B}_{J,n}^{(3)}$ is the $n$-th bicomplex third-order Jacobsthal quaternion, then the expression of the d'Ocagne's identity is given by the following:
\begin{align*}
\emph{\B}_{J,m}^{(3)}\times \emph{\B}_{J,n+1}^{(3)}-\emph{\B}_{J,m+1}^{(3)}\times \emph{\B}_{J,n}^{(3)}=\frac{1}{7}\left\lbrace
\begin{array}{c} 
\widehat{2}2^{m+1}\B_{U,n+1}^{(3)}-\widehat{2}2^{n+1}\B_{U,m+1}^{(3)}\\
+\widehat{\omega_{1}}\widehat{\omega_{2}}U_{m-n}^{(3)}
\end{array}
\right\rbrace,
\end{align*}
where $\emph{\B}_{U,n}^{(3)}=U_{n}^{(3)}+iU_{n+1}^{(3)}+jU_{n+2}^{(3)}+ijU_{n+3}^{(3)}$ and $U_{n}^{(3)}=\frac{1}{7}(2V_{n}^{(3)}-V_{n+1}^{(3)})$ as in (\ref{e6}), and $\widehat{2}$, $\widehat{\omega_{1}}$ and $\widehat{\omega_{2}}$ are the bicomplex quaternions defined in Theorem \ref{t4}.
\end{proposition}
\begin{proof}
Once more, using the Binet Formula of Theorem \ref{t4}, (\ref{e6}) and the fact that $\omega_{1}\omega_{2}=1$, we have
\begin{align*}
\B_{J,n}^{(3)}&=\frac{1}{\phi(J)}\left(\left(\omega_{1}-\omega_{2}\right) \widehat{2}2^{n+1}-\left(2-\omega_{2}\right) \widehat{\omega_{1}}\omega_{1}^{n+1}+\left(2-\omega_{1}\right) \widehat{\omega_{2}}\omega_{2}^{n+1}\right)\\
&=\frac{1}{7}\left(\widehat{2}2^{n+1}-\B_{V,n}^{(3)}\right).
\end{align*}
Then,
\begin{align*}
&\emph{\B}_{J,m}^{(3)}\times \emph{\B}_{J,n+1}^{(3)}-\emph{\B}_{J,m+1}^{(3)}\times \emph{\B}_{J,n}^{(3)}\\
&=\frac{1}{49}\left\lbrace
\begin{array}{c} 
\left(\widehat{2}2^{m+1}-\B_{V,m}^{(3)}\right)\left(\widehat{2}2^{n+2}-\B_{V,n+1}^{(3)}\right)\\
-\left(\widehat{2}2^{m+2}-\B_{V,m+1}^{(3)}\right)\left(\widehat{2}2^{n+1}-\B_{V,n}^{(3)}\right)\end{array}
\right\rbrace\\
&=\frac{1}{49}\left\lbrace
\begin{array}{c} 
-\widehat{2}2^{m+1}\left(\B_{V,n+1}^{(3)}-2\B_{V,n}^{(3)}\right)+\widehat{2}2^{n+1}\left(\B_{V,m+1}^{(3)}-2\B_{V,m}^{(3)}\right)\\
+\B_{V,m}^{(3)}\B_{V,n+1}^{(3)}-\B_{V,m+1}^{(3)}\B_{V,n}^{(3)}\end{array}
\right\rbrace\\
&=\frac{1}{7}\left(
\widehat{2}2^{m+1}\B_{U,n+1}^{(3)}-\widehat{2}2^{n+1}\B_{U,m+1}^{(3)}+\widehat{\omega_{1}}\widehat{\omega_{2}}U_{m-n}^{(3)}
\right),
\end{align*}
where $$\B_{V,n}^{(3)}=\left\{ 
\begin{array}{ccc}
2-3i+j+2ij & \textrm{if} & \mymod{n}{0}{3} \\ 
-3+i+2j-3ij & \textrm{if} & \mymod{n}{1}{3} \\ 
1+2i-3j+ij & \textrm{if} & \mymod{n}{2}{3}
\end{array}
\right. ,$$
and now, using the Binet formula for the numbers $U_{n}^{(3)}=\frac{1}{7}(2V_{n}^{(3)}-V_{n+1}^{(3)})$ as in (\ref{e6}), the result follows.
\end{proof}
 
For $m=n+1$ in d'Ocagne's identity, the Cassini-like identity for this sequence is stated in the next result.
\begin{proposition}[Cassini-like identity]
For a natural number $n$, if $\emph{\B}_{J,n}^{(3)}$ is the $n$-th bicomplex third-order Jacobsthal quaternion, then the following identity is true:
\begin{align*}
\left(\emph{\B}_{J,n+1}^{(3)}\right)^{2}-\emph{\B}_{J,n+2}^{(3)}\times \emph{\B}_{J,n}^{(3)}=\frac{1}{7}\left\lbrace
\begin{array}{c} 
\widehat{2}2^{n+1}\left(2\B_{U,n+1}^{(3)}-\B_{U,n+2}^{(3)}\right)+\widehat{\omega_{1}}\widehat{\omega_{2}}
\end{array}
\right\rbrace,
\end{align*}
where $\widehat{2}$, $\widehat{\omega_{1}}$ and $\widehat{\omega_{2}}$ are the bicomplex quaternions defined in Theorem \ref{t4}.
\end{proposition}

\begin{proposition}
For natural number $n$, if $\emph{\B}_{J,n}^{(3)}$ is the $n$-th bicomplex third-order Jacobsthal quaternion, then the following identity is true:
$$\left(\emph{\B}_{J,n}^{(3)}\right)^{2}+\left(\emph{\B}_{J,n+1}^{(3)}\right)^{2}+\left(\emph{\B}_{J,n+2}^{(3)}\right)^{2}=\frac{1}{7}\left(
3\cdot \widehat{2}^{2}2^{2n+2}-\widehat{2}2^{n+2}\emph{\B}_{U,n}^{(3)}+2ij
\right),$$
where $\emph{\B}_{U,n}^{(3)}=U_{n}^{(3)}+iU_{n+1}^{(3)}+jU_{n+2}^{(3)}+ijU_{n+3}^{(3)}$ and $U_{n}^{(3)}=\frac{1}{7}(2V_{n}^{(3)}-V_{n+1}^{(3)})$ as in (\ref{e6}), and $\widehat{2}$, $\widehat{\omega_{1}}$ and $\widehat{\omega_{2}}$ are the bicomplex quaternions defined in Theorem \ref{t4}.
\end{proposition}
\begin{proof}
By the use of the Binet Formula of Theorem \ref{t4}, the fact that the product of any two bicomplex third-order Jacosbthal quaternions is commutative and the identities $\omega_{1}\omega_{2}=1$ and $\omega_{1}+\omega_{2}=-1$, we have
\begin{align*}
&\left(\emph{\B}_{J,n}^{(3)}\right)^{2}+\left(\emph{\B}_{J,n+1}^{(3)}\right)^{2}+\left(\emph{\B}_{J,n+2}^{(3)}\right)^{2}\\
&=\frac{1}{49}\left\lbrace
\begin{array}{c} 
\left(\widehat{2}2^{n+1}-\emph{\B}_{V,n}^{(3)}\right)^{2}+\left(\widehat{2}2^{n+2}-\emph{\B}_{V,n+1}^{(3)}\right)^{2}+\left(\widehat{2}2^{n+3}-\emph{\B}_{V,n+2}^{(3)}\right)^{2}
\end{array}
\right\rbrace\\
&=\frac{1}{49}\left\lbrace
\begin{array}{c} 
\widehat{2}^{2}2^{2n+2}-\widehat{2}2^{n+2}\emph{\B}_{V,n}^{(3)}+\widehat{2}^{2}2^{2n+4}-\widehat{2}2^{n+3}\emph{\B}_{V,n+1}^{(3)}+\widehat{2}^{2}2^{2n+6}\\
-\widehat{2}2^{n+4}\emph{\B}_{V,n+2}^{(3)}+\left(\emph{\B}_{V,n}^{(3)}\right)^{2}+\left(\emph{\B}_{V,n+1}^{(3)}\right)^{2}+\left(\emph{\B}_{V,n+2}^{(3)}\right)^{2}
\end{array}
\right\rbrace\\
&=\frac{1}{49}\left\lbrace
\begin{array}{c} 
21\cdot \widehat{2}^{2}2^{2n+2}-\widehat{2}2^{n+2}\left(\emph{\B}_{V,n}^{(3)}+2\emph{\B}_{V,n+1}^{(3)}+4\emph{\B}_{V,n+2}^{(3)}\right)+14ij
\end{array}
\right\rbrace\\
&=\frac{1}{7}\left(
3\cdot \widehat{2}^{2}2^{2n+2}-\widehat{2}2^{n+2}\emph{\B}_{U,n}^{(3)}+2ij
\right).
\end{align*}
Now, using the Binet formula (\ref{e6}) and (\ref{ev}), the result follows.
\end{proof}

\section{Four-diagonal matrix with bicomplex third-order quaternions}
Now, we present another way to obtain the $n$-th term of the bicomplex third-order quaternion sequence as the computation of a tridiagonal matrix. We shall adopt the ideas stated in \cite[p. 277]{Cer} using the following result which is stated in such work:
\begin{theorem}\label{t6}
Let $\{x_{n}\}_{n\geq1}$ be any third-order linear sequence defined recursively by the following: $$x_{n+3}=rx_{n+2}+sx_{n+1}+tx_{n},\  n\geq 0,$$ with $x_{0}=A\neq 0$, $x_{1}=B$ and $x_{2}=C$. Then, for all $n\geq 0$:
\begin{equation} \label{f1}
x_{n}= \left| \begin{matrix}
A  & 1     & 0      &   \cdots & \cdots     & \cdots     & 0  \\
Ar-B  & r & \frac{1}{A} & 0 &\cdots &\cdots        & 0 \\
0 & Br-C & r & t & \ddots & \cdots & 0    \\
0  &A & -\frac{s}{t}  & r & t& \ddots & 0  \\
0  &0 & \frac{1}{t}  & -\frac{s}{t} & r & &   \\
\vdots &  \vdots &     & \ddots & \ddots & \ddots & t \\
0 & \cdots    & &     0   & \frac{1}{t}     & -\frac{s}{t}     & r\end{matrix} \right|_{(n+1)\times (n+1)}
\end{equation}
\end{theorem}

In the case of the bicomplex third-order quaternions sequence and taking into account the recurrence relation (\ref{e8}), we have $r=s=1$, $t=2$, and initial conditions $\B_{J,0}^{(3)}=i+j+2ij$, $\B_{J,1}^{(3)}=1+i+2j+5ij$ and $\B_{J,2}^{(3)}=1+2i+5j+9ij$.Then, by the use of the previous theorem, we have the following result which gives a different way to calculate the $n$-th term of this sequence:
\begin{proposition}
For $n\geq 0$, we have 
\begin{equation} \label{f2}
\emph{\B}_{J,n}= \left| \begin{matrix}
\emph{\B}_{J,0}  & 1     & 0      &   \cdots & \cdots     & \cdots     & 0  \\
\emph{\B}_{J,0}r-\emph{\B}_{J,1}  & 1 & \left(\emph{\B}_{J,0}\right)^{-1} & 0 &\cdots &\cdots        & 0 \\
0 & \emph{\B}_{J,1}r-\emph{\B}_{J,2} & 1 & 2 & \ddots & \cdots & 0    \\
0  &\emph{\B}_{J,0} & -\frac{1}{2}  & 1 & 2 & \ddots & 0  \\
0  &0 & \frac{1}{2}  & -\frac{1}{2} & 1 & &   \\
\vdots &  \vdots &     & \ddots & \ddots & \ddots & 2 \\
0 & \cdots    & &     0   & \frac{1}{2}     & -\frac{1}{2}     & 1\end{matrix} \right|_{(n+1)\times (n+1)}
\end{equation}
\end{proposition}

\section{Conclusions}
In this paper, the sequence of bicomplex third-order Jacobsthal quaternions defined by a recurrence relation of third order was introduced. Some properties involving this sequence, including the Binet formula and the generating function, were presented. Considering four diagonal matrices whose entries are bicomplex third-order Jacobsthal quaternions are presented a different way to obtain the $n$-th term of the bicomplex third-order Jacobsthal quaternions sequence.

\medskip
\end{document}